\documentclass[a4paper]{article}
\usepackage{latexsym}
\usepackage{amsmath,amssymb,amsfonts}
\usepackage{xcolor}

\newtheorem{prethm}{{\bf Theorem}}

\newenvironment{thm}{\begin{prethm}{\hspace{-0.5
               em}{\bf.}}}{\end{prethm}}

\newtheorem{prepro}[prethm]{Proposition}
\newenvironment{pro}{\begin{prepro}{\hspace{-0.5
               em}{\bf.}}}{\end{prepro}}

\newtheorem{prelem}[prethm]{Lemma}
\newenvironment{lem}{\begin{prelem}{\hspace{-0.5
               em}{\bf.}}}{\end{prelem}}

\newtheorem{precor}[prethm]{Corollary}
\newenvironment{cor}{\begin{precor}{\hspace{-0.5
               em}{\bf.}}}{\end{precor}}

\newtheorem{preremark}{{\bf Remark}}

\newtheorem{preexample}{{\bf Example}}

\newenvironment{example}{\begin{preexample}\em{\hspace{-0.5
               em}{\bf.}}}{\end{preexample}}

\newtheorem{predefinition}{{\bf Definition}}

\newenvironment{definition}{\begin{predefinition}\em{\hspace{-0.5
               em}{\bf.}}}{\end{predefinition}}

\newtheorem{preproof}{{\bf Proof.}}

\newenvironment{proof}[1]{\begin{preproof}{\rm
               #1}\hfill{$\Box$}}{\end{preproof}}

\input{epsf}

\author{{\normalsize { N. Ghareghani${}^{ \textrm{b},\,1}$}, { M. Mohammad-Noori${}^{ \textrm{a}}$}, {P. Sharifani ${}^{ \textrm{c}}$}\,
}\vspace{2mm} \\{\footnotesize{$^{ \textrm{a}}$\it
Department of
Mathematics, Statistics and Computer
Science, University of Tehran,}}\vspace{-2mm}\\{\footnotesize{\it   Tehran,
Iran}}\\
{\footnotesize{$^{ \textrm{b}}$\it
Department of
Industrial design, College of Fine Arts, University of Tehran,}}\\{\footnotesize{\it   Tehran,
Iran}}\\
{\footnotesize{$^{ \textrm{c}}$\it
Department of
Mathematics, Statistics and Computer
Science, University of Tehran,}}\\{\footnotesize{\it   Tehran,
Iran}}\\
\\{\footnotesize Emails: ghareghani@ut.ac.ir, ghareghani@ipm.ir},\vspace{-2mm}\\
 {\footnotesize  mmnoori@ut.ac.ir, pouyeh.sharifani@gmail.com.}
 }

\title{\bf\ Some Properties of the $k$-bonacci Words on Infinite Alphabet}
\date{}
\begin{document}

\maketitle
\footnotetext[1]{\tt Corresponding author}




\begin{abstract}
The Fibonacci word $W$ on an infinite alphabet was introduced in [Zhang et al., Electronic J. Combinatorics 2017 24(2), 2-52] as a fixed point of the morphism
$2i\rightarrow (2i)(2i+1)$, $(2i+1) \rightarrow (2i+2)$, $i\geq 0$.
Here, for any integer $k>2$, we define the infinite $k$-bonacci word $W^{(k)}$ on the infinite alphabet as the fixed point of the morphism
$\varphi_k$ on the alphabet $\mathbb{N}$ defined for any $i\geq 0$ and any $0\leq j\leq k-1$, as
\begin{equation*}
\varphi_k(ki+j) = \left\{
\begin{array}{ll}
(ki)(ki+j+1) & \text{if } j = 0,\cdots ,k-2,\\
(ki+j+1)& \text{otherwise}.
\end{array} \right.
\end{equation*}
We consider the sequence of finite words $(W^{(k)}_n)_{n\geq 0}$, where  $W^{(k)}_n$ is the prefix of $W^{(k)}$ whose length is the $(n+k)$-th $k$-bonacci number. We then provide a recursive formula for the number of palindromes occur in different positions of $W^{(k)}_n$.
Finally, we obtain the structure of all palindromes occurring in $W^{(k)}$ and based on this, we compute the palindrome complexity of $W^{(k)}$, for any $k>2$.
\end{abstract}
\vspace{3mm}

\section{Introduction}
Finite and infinite Fibonacci words are among the most studied ones in combinatorics of words and have important roles in computer science, based on their optimal properties and various applications, see for example \cite{morse1940symbolic, berstel1986fibonacci, tan2007some, de1981combinatorial}. The sequence of finite Fibonacci words $(F_n)_{n\geq -1}$ is given by $F_{-1}=1, F_0=0$ and the recurrence relation $F_n=F_{n-1}F_{n-2}$ which holds for $n\geq 1$.
An equivalent way to give these words for $n\geq 0$, is using $F_n=\psi^n(0)$, where $\psi$ is the binary morphism $0\rightarrow 01$, $1\rightarrow 0$. The infinite Fibonacci word is then given by $F_{\infty}={\displaystyle \lim _{n\rightarrow \infty}}F_n$ or equivalently by $F_{\infty}=\psi^{\omega}(0)$.

A natural extension of finite Fibonacci words to $k$-letter alphabet, $k\geq 2$, is defining finite $k$-bonacci words $(F_n^{(k)})_{n\geq 0}$ by
\begin{equation*}\label{defFrec}
F_n^{(k)} = \left\{
\begin{array}{ll}
0 \;\; &\text{if } \,\, n = 0,\\
F_{n-1}^{(k)}\ldots F_0^{(k)} n \;\; &\text{if } \,\, 1\leq n <k,\\
F_{n-1}^{(k)}\ldots F_{n-k}^{(k)} \;\; &\text{if } \,\, n\geq k.
\end{array} \right.
\end{equation*}
 Alternatively, these words may be given by $F_n^{(k)}=\psi_k(0)$, for $n\geq 0$, where $\psi_k: {\mathbb{N}}^* \rightarrow {\mathbb{N}}^*$ is the morphism
\begin{equation}\label{psidef}
\psi_k(i) = \left\{
\begin{array}{ll}
0 (i+1)\;\; &\text{if } \,\, i = 0,\cdots ,k-2, \\
0\;\; &\text{if } \,\, i = k-1.
\end{array} \right.
\end{equation}
 The infinite $k$-bonacci word is then given by $F^{(k)}_{\infty}={\displaystyle \lim _{n\rightarrow \infty}}F^{(k)}_n$ or equivalently by $F^{(k)}_{\infty}=\psi_k^{\omega}(0)$. We remark that another way to define the finite $k$-bonacci words, is using $F_n^{(k)}=\psi_k^n(0)$.

The infinite Fibonacci word belongs to the class of infinite aperiodic binary words having the minimal complexity (i.e. the minimal number of factors of each given length); Any such word is called a Sturmian word. Sturmian words are well-studied in the literature; They admit some equivalent definitions and have several interesting properties, see \cite{berstel2007sturmian, justin2000return, justin1996combinatorial} for instance.

While the infinite Fibonacci word is the simplest example of Sturmian words, the infinite $k$-bonacci word is similarly related to the most natural extension of Sturmian words, namely episturmian words. More precisely, the $k$-bonacci word is the simplest example of non-ultimately periodic episturmian words
 on the $k$-letter alphabet; To see the definition and properties of episturmian words see \cite{droubay2001episturmian, berstel2007sturmian, justin2002episturmian, glen2009episturmian, fischler2006palindromic}.

The infinite Fibonacci word over the infinite alphabet of nonnegative integers, $\mathbb{N}$, denoted here as $W^{(2)}$, is presented in \cite{zhang2017some} as the fixed point of the morphism $\varphi_2$ given by $\varphi_2(2i)=(2i)(2i+1)$ and $\varphi_2(2i+1)=2i+2$ for $i\geq 0$. More precisely, we have $W^{(2)}=\varphi_2^{\omega}(0)$. The authors of \cite{zhang2017some} have also studied the finite Fibonacci words over $\mathbb{N}$, namely the words $W_n^{(2)}=\varphi_2^{n}(0)$. It is obvious that when the digits of $W_n^{(2)}$ and $W^{(2)}$ are calculated$\mod 2$, these words are reduced to
$F_n$ and $F_{\infty}$, respectively. Among several properties of words $W_n^{(2)}$ and $W^{(2)}$ studied in \cite{zhang2017some}, the authors  characterized palindrome factors of $W_n^{(2)}$ and $W^{(2)}$. Particularly, the authors showed that in contrast to the ordinary infinite Fibonacci word which contains palindrome factors of arbitrary length, the word $W^{(2)}$ has no palindrome of length greater than $3$. Some more properties of these words were consequently studied by Glen et al. in \cite{glen2019more}. Among other results, they computed the number of palindromes in $W_n^{(2)}$.

In this paper, we introduce finite and infinite $k$-bonacci words on infinite alphabet $\mathbb{N}$, denoted respectively as $W_n^{(k)}$ and $W^{(k)}$. We study some properties of these words as well as the ones of the generating morphism $\varphi_k$, defined later in this paper. Based on these, we characterize the palindrome factors of $W^{(k)}$ for any fixed integer $k\geq 3$. More precisely we show that the length of a palindrome factor of $W^{(k)}$ belongs to the set
$L_k=\{2,3,5,\ldots, 3.2^{k-1}-1\}$. Conversely, for element $\ell$ of $L_k$ we give the structure of palindromes of  $W^{(k)}$ with length $\ell$. We also
enumerate the total number of palindromes of $W_n^{(k)}$.

\section{Definitions and notation}
In this paper, the alphabet, which can be a finite or a countable infinite set, is denoted as ${\mathcal A}$. When the alphabet is infinite,
we simply take ${\mathcal A}=\mathbb{N}$. Each element of the alphabet ${\mathcal A}$ is called a {\it letter}; When ${\mathcal A}=\mathbb{N}$, we equivalently use the term {\it digit} instead of letter. We denote by ${\mathcal A}^*$ the set of finite words over
${\mathcal A}$ and we let ${\mathcal A}^+ ={\mathcal A}^* \setminus \{\epsilon\}$, where $\epsilon$ the empty word. We denote by ${\mathcal A}^{\omega}$ the set of all infinite words over ${\mathcal A}$ and we let ${\mathcal A}^{\infty} ={\mathcal A}^*\cup {\mathcal A}^{\omega}$.
If $a\in{\mathcal A}$ and $W \in {\mathcal A}^{\infty}$, then the symbols $|W|$ and ${|W|}_a$ denote respectively the length of $W$, and the number of
occurrences of letter $a$ in $W$ (It is obvious that when $W\in {\mathcal A}^{\omega}$, $|W|=\infty$). For a word $W \in {\mathcal A}^{\infty}$, $\mathcal{A}lph(W)$ is defined to be the number of letters which have at least one occurrence in $W$, that is $\mathcal{A}lph(W)=\{a \in  {\mathcal A}: |W|_a>0\}$.

A word $V\in {\mathcal A}^+$
is a factor of a word $W \in {\mathcal A}^{\infty}$, denoted as $V \prec W$, if there exist
$U\in {\mathcal A}^*$ and $U' \in {\mathcal A}^{\infty}$, such that $W=UVU'$.
 A word $V\in {\mathcal A}^+$ (resp. $V\in {\mathcal A}^{\infty}$) is said to
be a {\textit prefix} (resp. {\textit suffix}) of a word $W\in {\mathcal A}^{\infty}$, denoted as $V\lhd W$ (resp.
$V\rhd W$), if there exists $U\in {\mathcal A}^{\infty}$ (resp. $U\in {\mathcal A}^{*}$) such that $W=VU$
(resp. $W=UV$). We denote the suffix (resp. prefix) $V$ of size $j$ of $W\in A^+$ by ${\rm Suff}_j(W)$ (resp. ${\rm Pref}_j(W)$). If $W\in {\mathcal A}^{*}$ and $W=VU$ (resp. $W=UV$,) we merely write
$V=WU^{-1}$ (resp. $V=U^{-1}W$).
For a finite word $W=w_1 w_2 \ldots w_{n}$, with $w_i \in {\mathcal A}$ and for $1\leq j\leq j'\leq n$, we denote $W[j,j']=w_j \ldots w_{j'}$, and for simplicity we denotes $W[j,j]$ by $W[j]$.
The {\textit
reversal} of a finite word $W=w_1 w_2 \ldots w_{n}$, with $w_i \in {\mathcal A}$,
 is $W^{R}=w_n w_{n-1}\ldots w_1$.
 A word $W\in {\mathcal A}^{*}$ is called {\textit palindrome} if
$W=W^{R}$. The set of all palindrome factors of the word $W\in {\mathcal A}^{\infty}$ is denoted by $\mathcal{P}al(W)$. When the alphabet is finite, for any word $U\in {\mathcal A}^{\infty}$, the number of palindrome factors of length $n$ of $U$, called the {\it palindrome complexity} of $U$, is denoted by ${\rm pal}_U(n)$ (for more information about the palindrome complexity see \cite{allouche2003palindrome, ambrovz2006palindromic} and the references therein).

 When the alphabet is infinite (i.e. ${\mathcal A}=\mathbb{N}$), the definition of palindrome complexity can naturally be extended to all words $U\in {\mathcal A}^{\infty}$ with the same notation. Let $P$ be a palindrome of odd length, then the letter $a\in {\mathcal A}$ for which $P=WaW^R$, is called the center of the palindrome $P$. If $P$ is a palindrome of even length, then the center of $P$ is defined to be the empty word.
 For any occurrence of a palindrome factor $P \in {\mathcal A}^*$ in a word $W \in {\mathcal A}^{\infty}$ such as $W=UPV$ with $U \in {\mathcal A}^*$, $V \in {\mathcal A}^{\infty}$, the center position of this occurrence of $P$ in $W$ is denoted by $c_p(P, W)$ and defined to be $|U|+\frac{|P|+1}{2}$. A palindrome factor $P\in {\mathcal A}$  of $W\in {\mathcal A}^*$ is called a {\it maximal palindrome factor} of $W$ if there is no longer palindrome factor of $W$ with the same center position.

For $1 \leq i \leq n$, let $U_i\in {\mathcal A}^{*}$; Then $\prod_{i=n}^{1} U_i$ is defined to be $U_n U_{n-1} \ldots   U_1$.
For a finite word $W$ and an integer $n$,  $n \oplus W$ denotes the word obtained by adding $n$ to each letter of $W$. For example, let $W=01023$ and $n=5$, then $n \oplus W=56578$. For a finite set $S=\{S_1, \ldots, S_m\}\subset {\mathcal A}^{+}$, we define $n\oplus S$ to be the set $\{n\oplus S_1, \ldots, n\oplus S_m\}$.

For any integer $k\geq 2$ the sequence of $k$-bonacci numbers, denoted by $\{f_n^{(k)}\}_{n\geq 0}$, is given as below
\begin{equation}\label{defkbonum}
f_n^{(k)} = \left\{
\begin{array}{ll}
0\;\; &\text{if } \,\, n = 0,\cdots ,k-2, \\
1\;\; &\text{if } \,\, n = k-1, \\
\sum_{i=n-1}^{n-k}f_i^{(k)}\;\; &\text{if}  \,\, n\geq k.
\end{array} \right.
\end{equation}
The sequences corresponding to $k=2$ and $k=3$ are respectively the well-known Fibonacci and Tribonacci sequences.
We define the finite (resp. infinite) $k$-bonacci words $W^{(k)}_n$ (resp. $W^{(k)}$) on infinite alphabet $\mathbb{N}$, using the morphism $\varphi_k$ given below
\begin{equation*}
\varphi_k(ki+j) = \left\{
\begin{array}{rl}
(ki)(ki+j+1)\;\; & \text{if } j = 0,\cdots ,k-2 \\
(ki+j+1)& \text{otherwise } .
\end{array} \right.
\end{equation*}
More precisely, $W^{(k)}_n=\varphi_k^n(0)$ and $W^{(k)}=\varphi_k^{\omega}(0)$ (Note that $W^{(k)}_0=F^{(k)}_0=0$).
  Consequently $F^{(k)}_n= W_n^{(k)}  \mod\; k$,  that is for a fixed value of $k$, the $k$-bonacci words over infinite alphabet are reduced to $k$-bonacci words over finite alphabet when the digits are calculated $\mod\; k$. It is easy to show that for $n \geq 0$,
\begin{equation}\label{size}
|F^{(k)}_n|= |W_n^{(k)}|=f_{n+k}^{(k)}.
\end{equation}
We end this section with two examples giving some initial $k$-bonacci words on finite and infinite alphabet.
\begin{example}\label{infkboexam} In this example the third $k$-bonacci words on finite and infinite alphabet are given when $k\leq 5$. The infinite Tribonacci words on finite and infinite alphabet are also shown.
 \begin{align*}
 &W_0^{(3)} = 0,
 &&F_0^{(3)} =0,\\
 &W_1^{(3)} = 01,
 &&F_1^{(3)} =01,\\
 &W_2^{(3)} =0102,
 &&F_2^{(3)} =0102,\\
 &W_3^{(3)} =0102013,
 &&F_3^{(3)} =0102010,\\
 &W_4^{(3)} =0102013010234,
 &&F_4^{(3)} =0102010010201,\\
 &W_5^{(3)} =010201301023401020133435,
 &&F_5^{(3)} =010201001020101020100102,\\
 \end{align*}
  The first  terms of $W^{(3)} $ and $\mathbf{F}^{(3)} $ are as below
  \begin{align*}
  W^{(3)} &=0\;1\;0\;2\;0\;1\;3\;0\;1\;0\;2\;0\;1\;3\;4\;0\;1\;0\;2\;0\;1\;3\;0\;1\;0\;2\;0\;1\;3\;4\; 0\;1\;0\;2\;0\;1\;3\;0\;1\;0\;2\;0\;1\;3\;4\;3\;5\;\cdots\\
   \mathbf{F}^{(3)} &= 0\;1\;0\;2\;0\;1\;0\;0\;1\;0\;2\;0\;1\;0\;1\;0\;1\;0\;2\;0\;1\;0\;0\;1\;0\;2\;0\;1\;0\;1\; 0\;1\;0\;2\;0\;1\;0\;0\;1\;0\;2\;0\;1\;0\;1\;0\;2\;\cdots\\
  \end{align*} 	
\end{example}
\begin{example} In this example we fix $n=6$ and consider different cases $k=3,4,5,6$.
 \begin{align*}
&W_6^{(3)} =01020130102340102013343501020130102343435346,\\
&F_6^{(3)} =01020100102010102010010201020100102010102010.\\
\\
&W_6^{(4)} =01020103010201401020103010245010201030102014010201034546,\\
&F_6^{(4)} =01020103010201001020103010201010201030102010010201030102.\\
\\
&W_6^{(5)} =0102010301020104010201030102015010201030102010401020103010256,\\
&F_6^{(5)} =0102010301020104010201030102010010201030102010401020103010201.\\
\\
&W_6^{(6)} =010201030102010401020103010201050102010301020104010201030102016,\\
&F_6^{(6)} =010201030102010401020103010201050102010301020104010201030102010.
\end{align*}
\end{example}
\section{Some properties of the morphism $\varphi_k$}
In this section we give some basic properties of the morphism  $\varphi_k$. We will later use these properties to discover the structure of palindromes in $k$-bonacci words. The following lemma states a property of the morphism $\varphi_k$ which can be easily deducted from the definition.

\begin{lem} \label{phideflem}
	For any integer $i\geq 0$,   $\varphi_k(k + i)=k\oplus \varphi_k(i)$.
\end{lem}

\begin{thm}\label{growth}
	For any $n \geq 1$, we have
	$\varphi_k^n(ki + j) = \varphi_k^n(j) \oplus (ki)$
	for any $i, j \geq 0$.
\end{thm}
\begin{proof}
	{We use induction on $n$.
		When $n = 1$, the result follows from the definition of $\varphi_k$. Now suppose the result
		holds for all $n$, $1\leq n \leq m$; We will prove it for $n=m + 1$.

		\begin{itemize}
			\item If $j \equiv k-1 (\mod k)$, then		
			for any $i \geq 0$,
		\begin{align*}
		\varphi_k^{m+1}(ki + j)= &\varphi_k^{m}(ki + j+1)\\
		=&\varphi_k^{m}(j+1)\oplus (ki)\text{ by the induction hypothesis} \\
		=&\varphi_k^{m+1}(j)\oplus (ki).
		\end{align*}
		\item If $j \not\equiv k-1 (\mod k)$, then let $j=kj'+j''$ with $0\leq j'' \leq k-1$. For any $i \geq 0$,
		\begin{align*}
		\varphi_k^{m+1}(ki + j)= &\varphi_k^{m}((ki + kj') (ki + j+1))\\
	     = &\varphi_k^{m}(ki + kj') \varphi_k^{m}(ki + j+1)\\
		=&(\varphi_k^{m}(kj')\oplus (ki)) (\varphi_k^{m}(j+1)\oplus (ki))\,\,\,\,\text{ by the induction hypothesis} \\
		=&(\varphi_k^{m}(kj') \varphi_k^{m}(j+1))\oplus (ki)\\
		=&(\varphi_k^{m}((kj') (kj'+j''+1)))\oplus (ki)\\
        =&\varphi_k^{m}(\varphi_k(j))\oplus (ki)\\
		=&\varphi_k^{m+1}(j)\oplus (ki).
		\end{align*}
		\end{itemize}}
\end{proof}
The following theorem gives us the suffix of length two of the word $\varphi_k^n(0)$. The proof is based on some definitions and Lemmas come in the next section. So, we demonstrate the proof of this theorem later.
\begin{thm}\label{suffix} Let $n \geq 1$ and
	 $n \equiv j\;(mod\; k)$ and $0 \leq j \leq k-1$. Then
	\begin{itemize}
		\item $(n-j) n \rhd \varphi_k^n(0)$ provided $j\neq 0$.
		\item $(n-k+1) n \rhd \varphi_k^n(0)$ provided $j=0$.
	\end{itemize}
\end{thm}
We end this section with the following simple lemma.
\begin{lem}\label{00}
    Let $n \geq 0$ and $k>2$. The finite word $W_n^{(k)}$ contains no factor $00$.
\end{lem}
\begin{proof}
	{
We prove the result by induction on $n$. By definition of $W_i^{(k)}$  for $0 \leq i \leq k-1$, it is obvious that there is no factor $00$ in $W_i^{(k)}$. Now, suppose that there is no factor $00$ in $W_i^{(k)}$ for $i \leq m$. We are going to prove that $00$ is not a factor of $W_{m+1}^{(k)}$. Since $W_{m+1}^{(k)} =W_m^{(k)}   \ldots   W_{m+2-k}^{(k)}   (k \oplus W_{m+1-k}^{(k)})$, by induction hypopiesis, $00$ is not a factor of $W_{m+1}^{(k)}$ that is not bordering or straddling factor. In addition, $W_i^{(k)}$ end by $i$, for $i \geq 1$. So $00$ is not a  bordering and straddling factor of $W_{m+1}^{(k)}$ and the claim is true.
}
\end{proof}

\section{Some properties of $W^{(k)}_{n}$}
In this section we give some identities for $W^{(k)}_{n}$ which are very useful for the rest of the paper.

The following lemma gives a recursive formula for $W^{(k)}_{n}$, when $1 \leq n \leq k-1$.

\begin{lem} \label{struct1}
	For $1 \leq n \leq k-1$,
	\begin{equation}\label{struct1eq}
		W^{(k)}_n=\prod_{i=n-1}^{0}W^{(k)}_i \, \, n.
	\end{equation}
\end{lem}
 \begin{proof}
 	{We use induction on $n$. It is easily to check that the statement true for $j = 1$. Suppose the lemma is true for all $i$
 		with $1 \leq i \leq j<k-1$. Then
 		\begin{align*}
 		W^{(k)}_{j+1}= &\varphi_k(W^{(k)}_{j})\\
 		=&\varphi_k(\prod_{i=j-1}^{0}W^{(k)}_i \,\, j)\\
 		=&\prod_{i=j-1}^{0} \varphi_k(W^{(k)}_i)\, \varphi_k(j) \\
 		=&\prod_{i=j}^{1}W^{(k)}_i \,\, 0 \, (j+1)\,\,\,\,\,\,\text{ since }  j <  k-1 \text{ and }  \varphi_k(j)=0 (j+1)\\
 		=&\prod_{i=j}^{0}W^{(k)}_i \, \, (j+1).
 		\end{align*}}
 \end{proof}

\begin{lem} \label{struct2}
	For $1 \leq n \leq k-1$,
	$$ W^{(k)}_n=W^{(k)}_{n-1} W^{(k)}_{n-1}(n-1)^{-1} n. $$
\end{lem}
 \begin{proof}
	{By Lemma \ref{struct1}, we have
		\begin{align*}
			W^{(k)}_n&=W^{(k)}_{n-1} W^{(k)}_{n-2}  \ldots   W^{(k)}_{0}  n\\
			&= W^{(k)}_{n-1}  W^{(k)}_{n-1}(n-1)^{-1} n		
		\end{align*}}
\end{proof}
	
In the next lemma we give a recursive formula for $ W^{(k)}_n$ when $n \geq k$.
\begin{lem} \label{struct}
	For $n \geq k$,
	\begin{equation}\label{structeq}
 W^{(k)}_n=\prod_{i=n-1}^{n-k+1}W^{(k)}_i \, (k\oplus W^{(k)}_{n-k}).
	\end{equation}
\end{lem}	
\begin{proof}{
We use induction on $n$. If $n=k$, the statement is true since
\begin{align*}
W^{(k)}_{k}= &\varphi_k(W^{(k)}_{k-1})\\
=&\varphi_k(\prod_{i=k-2}^{0}W^{(k)}_i (k-1))\\
=&\prod_{i=k-2}^{0} \varphi_k(W^{(k)}_i) \varphi_k(k-1) \\
=&\prod_{i=k-1}^{1} W^{(k)}_i k\\
=&\prod_{i=k-1}^{1} W^{(k)}_i (k\oplus W^{(k)}_0).
\end{align*}
 Now, suppose (\ref{structeq}) holds for all $n$
with $k\leq n \leq j$; We prove it for $n=j+1$ as below.
\begin{align*}
W^{(k)}_{j+1}= &\varphi_k(W^{(k)}_{j})\\
=&\varphi_k(\prod_{i=j-1}^{j-k+1}W^{(k)}_i (k\oplus W^{(k)}_{j-k}))\\
=&\prod_{i=j-1}^{j-k+1} \varphi_k(W^{(k)}_i) \varphi_k(k\oplus W^{(k)}_{j-k})\\
=&\prod_{i=j}^{j-k+2}W^{(k)}_i (k\oplus \varphi_k( W^{(k)}_{j-k}))\,\,\,\, \text{by Lemma \ref{phideflem}}\\
=&\prod_{i=j}^{j-k+2}W^{(k)}_i (k\oplus  W^{(k)}_{j-k+1}).
\end{align*}}
\end{proof}

\begin{cor}\label{correcsize}
Let $k>2$, then the following equality holds:
\begin{align*}
|W^{(k)}_{i+1}|= \left\{
\begin{array}{rl}
2|W^{(k)}_{i}|\;\; & \text{if } 0\leq i\leq k-2, \\
2|W^{(k)}_{i}|-1 & \text{if } i=k-1,\\
2|W^{(k)}_{i}|-|W^{(k)}_{i-k}| & \text{if } i>k-1.
\end{array} \right.
\end{align*}
\end{cor}
\begin{proof}
{The result is true either by (\ref{size}), or by Lemmas \ref{struct1} and \ref{struct}.
}
\end{proof}
The following corollary is a direct consequence of Corollary \ref{correcsize}.
\begin{cor}\label{coreqsize}
The following inequality holds for every nonnegative integer $i$.
$$|W^{(k)}_{i+1}|\leq 2 |W^{(k)}_{i}|$$
\end{cor}

\begin{lem}\label{lastdig}
	For any $n\geq 1$, the digit $n$ is the largest one of $W^{(k)}_n$ and appears once at the end of this word.
\end{lem}
\begin{proof}
	{We prove this using induction on $n$.
	When $n=1$, $W^{(k)}_1=\varphi_k(0)=01$ and it is obvious that the claim is true. Now suppose the result holds for $n \leq m$. We must it for $n=m+1$.
	\begin{align*}
	W^{(k)}_{m+1}= &\varphi_k(W^{(k)}_{m})\\
	=&\varphi_k(W^{(k)}_{m}m^{-1} m)\,\,\,\text{by the induction hypothesis}\\
	=&\varphi_k(W^{(k)}_{m}m^{-1}) \varphi_k(m).
	\end{align*}
Now, by induction hypothesis all digits of $W^{(k)}_{m}m^{-1}$ are less than $m$. Hence, by definition of $\varphi_k$, all digits of $\varphi_k(W^{(k)}_{m}m^{-1})$  are less than $m+1$. Again using definition of $\varphi_k$, the largest digit of $\varphi_k(m)$ is $m+1$ which occurs once at
the end of $\varphi_k(m)$ and the proof is complete.
}\end{proof}

\begin{lem}\label{ab<wn}
Let $a, b \in \mathbb{N}$ and let $n,k$ be two positive integers with $k>2$. If $ab \prec W^{(k)}_n$, then either $k\mid b$ or $a<b$.
\end{lem}
\begin{proof}
{We prove the lemma using induction on $n$, for a fixed integer $k>2$. If $n=1$, then $ W^{(k)}_1=01$ and it is clear that the lemma is true.
Let the claim be true for all $n$ with $n\leq m$; We want to prove it for $n=m+1$. Assume that $ab \prec W^{(k)}_{m+1}$ and $k\nmid b$. We show that $a<b$.

If $n \leq k-1$, then by (\ref{struct1eq}) we have
\begin{equation*}
 W^{(k)}_n=\prod_{i=n-1}^{0}W^{(k)}_i \, \, n.
 \end{equation*}
Since $W^{(k)}_i$ always starts with $0$ and $k\nmid b$, using above equation we conclude that either $ab\prec W^{(k)}_i$, for some $1\leq i\leq n-1$ or
$a=0$ and $b=n$. If $ab\prec W^{(k)}_i$, for some $1\leq i\leq n-1$, then by induction hypothesis $a<b$ and if $ab=0n$ it is clear that $a<b$, as required.

If $n \geq k$, then by (\ref{structeq}) we have
\begin{equation*}
 W^{(k)}_n=\prod_{i=n-1}^{n-k+1}W^{(k)}_i \, \, (k\oplus W^{(k)}_{n-k}).
	\end{equation*}
Since $W^{(k)}_i$ always start with $0$ and $k\nmid b$, using above equation we conclude that either $ab\prec W^{(k)}_i$, for some $n-k+1\leq i\leq n-1$ or
$ab\prec (k\oplus W^{(k)}_{n-k})$. In both cases by induction hypothesis we have $a<b$, as desired.
}
\end{proof}
	
\noindent {\bf{Proof of Theorem \ref{suffix}:}}
We prove the theorem using induction on $n$. Suppose the result holds for $n \leq m$. We consider two following cases for $n=m+1$:
	\begin{itemize}
		\item 	If $m+1 \equiv 0\;(mod \;k)$, then $m \equiv k-1\;(mod\; k)$ and by induction hypothesis $(m-k+1) (m) \rhd \varphi_k^m(0)$. Hence,
		\begin{align*}
			\varphi_k((m-k+1) (m)) \rhd & \varphi_k^{m+1}(0),\\
		\varphi_k(m-k+1) \varphi_k(m) \rhd & \varphi_k^{m+1}(0),\\
		(m-k+1)(m-k+2)(m+1)\rhd& \varphi_k^{m+1}(0).
		\end{align*}
	\item	If $m+1 \equiv j\;(mod \;k)$, and $0 <j<k$ then by Lemma \ref{lastdig}, $m$ is the last digit of $\varphi_k^m(0)$ or $m \rhd  \varphi_k^m(0)$. Hence,
		\begin{align*}
		\varphi_k(m) \rhd & \varphi_k^{m+1}(0)\\
		(m-j+1)(m+1)\rhd& \varphi_k^{m+1}(0).
		\end{align*}
\end{itemize}
So, the proof is complete.\hfill{$\Box$}

\section{ The number of palindromes in $W^{(k)}_n$}
In this section we are going to count the total number of palindromes in $W^{(k)}_n$.
 Actually, the process of counting all palindromes in $W^{(k)}_n$ leads us to find all possible palindrome factors of  $W^{(k)}_n$. Equations (\ref{struct1eq}) and (\ref{structeq}) are essential in the rest of this work. By equation (\ref{struct1eq}) and Lemma \ref{struct2}, we obtain an explicit formula for the number of palindromes in $W^{(k)}_n$ when $n \leq k-1$; This is done in Section \ref{2-n-k-1}. When $n\geq k$, equation (\ref{structeq}) leads to distinguish the three following types of palindromes.

\begin{description}
	\item[]  {\bf Type 1.} Palindromes which are included in one of the words  $W^{(k)}_i$ , $n-k+1\leq i\leq n-1$ or in $(k\oplus W^{(k)}_{n-k})$,
	\item[] {\bf Type 2.} Palindromes $P$ which are of the forme $P=X_jY_j$ for some  $n-k+2\leq j\leq n-1$ in which
	$X_j\neq \epsilon$ is a suffix of $W^{(k)}_j$ and $Y_j\neq \epsilon$ is a prefix of $\prod_{i=j-1}^{n-k+1}W^{(k)}_i$ for some $n-k+2\leq j\leq n-1$. We call any such palindrome, a {\it bordering palindromes of type $j$} of $W^{(k)}_n$.
	\item[] {\bf Type 3.} Palindromes $P$ which are of the forme $P=AB$, where
	$A$ is a suffix of $\prod_{i=n-1}^{n-k+1}W^{(k)}_i$ and $B$ is a prefix of
	$k\oplus W^{(k)}_{n-k}$, these are called {\it $(A,B)$-straddling palindromes} (or straddling palindromes for short, if there is no danger of confusion) of  $W^{(k)}_n$.
	\end{description}

	Let $P^{(k)}(n)$ denote the total number of palindromes in $W^{(k)}_n$ occurred in different positions and
 $B^{(k)}(n,j)$ and $S^{(k)}(n)$ denote the number of bordering palindromes of type $j$ and straddling palindromes of $W^{(k)}_n$, respectively.
It is clear that the following recurrence relation holds
\begin{equation}\label{recpbs}
 P^{(k)}(n)= \Sigma_{i=n-k}^{n-1} P^{(k)}(i)+ \Sigma_{i=n-k+2}^{n-1} B^{(k)}(n,i)+S^{(k)}(n).
\end{equation}


The following theorem gives a recurrence formula for computing $P^{(k)}(n)$. The proof which is formally stated in Section \ref{secproofthmpal}, is based on considering several cases with respect to values of $n$ and $k$. This is done in Sections \ref{2-n-k-1}-\ref{secstrad}.

\begin{thm}\label{recppthm}
Let $k>2$ and $n\geq 1$ be given integers. Then the following holds
\begin{description}
\item[(i)] If $1 \leq n \leq k-1$, then $P^{(k)}(n)=2^{n-1}(n-2)+1$,\\
\item[(ii)] If $n \geq k$, then $P^{(k)}(n)= \Sigma_{i=n-k}^{n-1} P^{(k)}(i)+\alpha^{(k)}(n)$, where

\begin{equation*}\label{recpp}
 \alpha^{(k)}(n)= \left\{
\begin{array}{ll}
2^k+(k-3) 2^{n-k+2}-n2^{n-k+1}& \text{if} \,\,\,\, k \leq n \leq 2k-3,\\
0 & \text{if} \,\,\,\, n=2k-2,\\
2^{n-2k+2}-1& \text{if} \,\,\,\, 2k-1 \leq n \leq 3k-3,\\
2^{k}-2& \text{if} \,\,\,\, n=3k-2,\\
0 & \text{if} \,\,\,\, n>3k-2.
\end{array} \right.
\end{equation*}
\end{description}
\end{thm}

To prove the theorem we divide this section in some subsections with respect to the values of $n$ and $k$.

\subsection{Palindromes in $W^{(k)}_n$ when $2\leq n \leq k-1$}\label{2-n-k-1}

\begin{lem}\label{wn<k}
	For $2\leq n \leq k-1$,	$W^{(k)}_n\, n^{-1}$ is a palindromic word.
\end{lem}
\begin{proof}
	{We prove this by induction on $n$. Since for every $k>2$ we have $W^{(k)}_2=0102$ the first step of the induction is true. Suppose $n=j<k-1$, the word $W^{(k)}_j\, j^{-1}$ is palindrome. Now using Lemma \ref{struct2} we have
		\begin{align*}
		W^{(k)}_{j+1}(j+1)^{-1}&=W^{(k)}_{j} W^{(k)}_{j}j^{-1}\\
		&=W^{(k)}_{j}j^{-1} j W^{(k)}_{j}j^{-1}
		\end{align*}
		which is a palindrome word by induction hypothesis.}
\end{proof}

\begin{lem} \label{initial1}
	For every $2 \leq n \leq k-1$, $P^{(k)}(n)=2P^{(k)}(n-1)+2^{n-1}-1$ and
 $P^{(k)}(1)=0$.	
\end{lem}

\begin{proof}
	{Since the digit $n$ just occurs in the last position of $W^{(k)}_n$,
		every palindrome factors of $W^{(k)}_n$ should be a palindrome factor of $W^{(k)}_n (n)^{-1}$. By Lemma \ref{struct2}, for every $n \leq k-1$, we have
\begin{equation}\label{initialeq}
W^{(k)}_n n^{-1}=W^{(k)}_{n-1} W^{(k)}_{n-1}(n-1)^{-1}=W^{(k)}_{n-1}(n-1)^{-1} (n-1) W^{(k)}_{n-1}(n-1)^{-1}
\end{equation}
From (\ref{initialeq}), we conclude that $n-1$ occurs once in $W^{(k)}_n(n)^{-1}$. Using Lemma \ref{wn<k} and again using (\ref{initialeq}), we  found that
a factor $p$ of $W^{(k)}_n(n)^{-1}$ is palindrome if and only if it is either a palindromic factor of $W^{(k)}_{(n-1)}$ or $p=a (n-1) a$, where $a$ is a nonempty prefix of  $W^{(k)}_{(n-1)}$.
Therefore, $P^{(k)}(n)=2P^{(k)}(n-1)+|W^{(k)}_{n-1}|-1= 2P^{(k)}(n-1)+ 2^{n-1}-1$.}
\end{proof}

\begin{thm}\label{initialthm}
	For every $1 \leq n \leq k-1$, we have the following explicit formula $P^{(k)}(n)=2^{n-1}(n-2)+1$.
\end{thm}

\begin{proof}
{We can easily prove the formula by induction on $n$.}
\end{proof}

\subsection{Bordering Palindromes of  $W^{(k)}_{n}$}\label{secborder}

In this section we consider the bordering palindromes of $W^{(k)}_n$.
 A bordering palindrome $B$ of $W^{(k)}_{n}$ is called a {\it maximal bordering palindrome} if there is no longer bordering palindrome factor of $W^{(k)}_{n}$
 with the same center position.

\begin{lem} \label{centerpal1}
	Let $k\leq n$ and $B_j$ be one of the bordering palindromes of type $j$ of $W^{(k)}_n$. Then
$c_p(B_j, W^{(k)}_n)=|W^{(k)}_j|$ and $j$ is the center of the palindrome $B_j$.
\end{lem}

\begin{proof}
	{By definition of bordering palindrome of type $j$ of $W^{(k)}_n$, we know that $B_j$
is a factor of $W=W^{(k)}_{j} W^{(k)}_{j-1} \ldots W^{(k)}_{n-k+1}$, and it contains the last digit of $W^{(k)}_{j}$ which is $j$. By Lemma \ref{lastdig} $|W|_j=1$, whence the result follows.}
\end{proof}

\begin{lem} \label{borderpal1}
	Let $k\leq n$ and $B_j$ be a bordering palindrome of type $j$ of $W^{(k)}_n$. Then $n\leq 2k-3$ and $n-k+2\leq j\leq k-1$.
\end{lem}

\begin{proof}
{Let $c=c_p(B_j, W^{(k)}_n)$. By Lemmas \ref{suffix} and  \ref{centerpal1}, we have $B_j[c]=j$ and $B_j[c-1]=j-j'$, where $j \equiv j'\;(mod\; k)$ and
$0\leq j'\leq k-1$. By equation (\ref{structeq}), the prefix $W^{(k)}_j$ of $W^{(k)}_n$, followed by $W^{(k)}_{j-1}$, whose first digit is $0$, hence
$B_j[c+1]=0$. On the other hand by definition of $c$, we obtain $B_j[c+1]=B_j[c-1]=j-j'$. Therefore, $j=j'$, which shows that
 $0\leq j\leq k-1$. On the other hand, by equation (\ref{structeq}), $n-k+2\leq j \leq n-1$. Hence $n-k+2\leq j\leq k-1$, which shows that $n\leq 2k-3$, as desired.}
\end{proof}

\begin{lem}\label{bj}
Let $k\leq n\leq 2k-3$ and $n-k+2\leq j\leq k-1$. Then the maximal bordering palindrome of type $j$ of
$W^{(k)}_n$ is $B_j= (W^{(k)}_{j-1}W^{(k)}_{j-2}\ldots W^{(k)}_{n-k+1})^R j(W^{(k)}_{j-1}W^{(k)}_{j-2}\ldots W^{(k)}_{n-k+1})$.
\end{lem}
\begin{proof}
{ By definition of bordering palindrome, $B_j$ is a palindromic factor of the following word
$$W^{(k)}_j j^{-1} j W^{(k)}_{j-1}W^{(k)}_{j-2}\ldots W^{(k)}_{n-k+1},$$ and by Lemma \ref{centerpal1}, $j$ is the center $B_j$. Now, by Lemma \ref{wn<k}, $W^{(k)}_j j^{-1}$ is palindrome. By Lemma \ref{borderpal1}, $j\leq k-1$ and $n-k+1\leq j-1$. On the other hand by Lemma \ref{struct1} we have $W^{(k)}_{j-1}W^{(k)}_{j-2}\ldots W^{(k)}_{n-k+1} \lhd W^{(k)}_j j^{-1}$. From these two points we conclude that
$$B_j=(W^{(k)}_{j-1}W^{(k)}_{j-2}\ldots W^{(k)}_{n-k+1})^R j(W^{(k)}_{j-1}W^{(k)}_{j-2}\ldots W^{(k)}_{n-k+1}).$$}
\end{proof}

\begin{lem}\label{lengthbj}
Let $n$ and $j$ be two integers with $n\leq 2k-3$ and $n-k+2\leq j\leq k-1$ and $B_j$ be a bordering palindrome of type $j$. Then $|B_j|=2(|W^{(k)}_{j}|-|W^{(k)}_{n-k+1}|)+1$.
\end{lem}
\begin{proof}
{By Lemma  \ref{bj} and (\ref{struct1eq}) the maximal bordering palindrome of type $j$ of $W^{(k)}_n$ is of the following form:
	\begin{align*}
B_j=& (W^{(k)}_{j-1}W^{(k)}_{j-2}\ldots W^{(k)}_{n-k+1})^R j(W^{(k)}_{j-1}W^{(k)}_{j-2}\ldots W^{(k)}_{n-k+1})\\
=& (W^{(k)}_{0}  \ldots W^{(k)}_{n-k})^{-1}W^{(k)}_{j} (j)^{-1} (j)  W^{(k)}_{j}(W^{(k)}_{n-k}  \ldots W^{(k)}_{0})^{-1}.
\end{align*}
Hence
\begin{align*}
|B_j|&=2 |W^{(k)}_{j}(W^{(k)}_{n-k}  \ldots W^{(k)}_{0})^{-1}|+1\\
     &=2|W^{(k)}_{j}(W^{(k)}_{n-k+1})^{-1}|+1\\
     &=2(|W^{(k)}_{j}|-|W^{(k)}_{n-k+1}|)+1.
\end{align*}
}
\end{proof}

\begin{lem} \label{bordernumber}
Let $k\leq n$, then
	\begin{equation*}
	B^{(k)}(n,j)= \left\{
	\begin{array}{rl}
	2^j-2^{n-k+1} & \text{if } (k \leq n \leq 2k-3)\,\,\, \text{and}\,\,\, (n-k+2\leq j \leq k-1),\\
	0\;\;\;\;\;\;\;\;& \text{otherwise }.
	\end{array} \right.
	\end{equation*}
	\end{lem}

\begin{proof}
	{If $n\leq 2k-3$ and $n-k+2\leq j\leq k-1$, then by Lemma \ref{lengthbj}, we have
	\begin{align*}
B^{(k)}(n,j)=&\dfrac{|B_j|-1}{2}\\
=&|W^{(k)}_{j}|-|W^{(k)}_{n-k+1}|\\
= &2^{j}-2^{n-k+1}
\end{align*}
Since  $k \leq n \leq 2k-3$ and  $j <k$, the last equality holds.
Otherwise, by Lemma \ref{borderpal1}, $B^{(k)}(n,j)=0$, as desired.}
\end{proof}

\subsection{Straddling Palindromes of $W_n^{(k)}$}\label{secstrad}
In this section we are going to count the number of straddling palindromes of $W_n^{(k)}$.
 A straddling palindrome $S$ is called an {\it maximal straddling palindrome} of $W_n^{(k)}$ if there is no longer straddling palindrome factor of $W^{(k)}_{n}$ with the same center position.
Similarly, an $(A,B)$-maximal straddling palindrome, is a  $(A,B)$-straddling palindrome which is a maximal straddling palindrome.

\begin{lem}\label{straddling}
	If $S$ is a straddling palindrome of $W_n^{(k)}$, then $2k-1 \leq n \leq 3k-2$.
\end{lem}

\begin{proof}
	{First we note that  $(k\oplus W^{(k)}_{n-k})$ starts with $k$ and  $W^{(k)}_{n-k+1}$ ends with $n-k+1$. Hence, the sequence $(n-k+1)k$ occurs in any straddling palindrome $P$ of  $W_n^{(k)}$. Since $P$ is palindrome, the sequence $k(n-k+1)$ also occurs in $P$. By definitions of $W_n^{(k)}$ and $\phi_k$
		the only letters that come after $k$ are from the set $\{k, k+1, k+2, \ldots, 2k-1\}$. Therefore, we have $(n-k+1)\in \{k, k+1, k+2, \ldots, 2k-1\}$, hence,
		$n\in \{2k-1, 2k, \ldots, 3k-2\}$, as desired.}
\end{proof}

\begin{lem}\label{n=2k-1}
	The only straddling palindrome in $W^{(k)}_{2k-1}$ is $kk$.
\end{lem}
\begin{proof}
	{By Lemma \ref{struct}, $W^{(k)}_{2k-1}=W^{(k)}_{2k-2}  \ldots  W^{(k)}_{k} (k \oplus W^{(k)}_{k-1})$. It is clear that $k\lhd (k \oplus W^{(k)}_{k-1})$ and using Theorem \ref{suffix}, $(1 k)\rhd W^{(k)}_{k}$ and hence $kk$ is a straddling palindrome of $W^{(k)}_{2k-1}$. Now, we are going to show that there is no other straddling palindrome in $W^{(k)}_{2k-1}$. 		
		Since there is no digit $1$ in $ (k \oplus W^{(k)}_{k-1})$, every straddling palindrome of  $W^{(k)}_{2k-1}$ has the last digit of  $W^{(k)}_{k}$ (i.e. $k$) as a prefix.
		By Lemma \ref{00}, $0 0$ is not a factor of $W^{(k)}_{k-1}$ and hence $kk$ is not a factor of $(k \oplus W^{(k)}_{k-1})$. Therefore, $W^{(k)}_{2k-1}$ could not have a straddling palindrome of length greater than $2$ and $kk$ is its only straddling palindrome.}
\end{proof}

\begin{lem}\label{strad1|A|}
Let $2k-1 < n < 3k-2$ and let $S=XY$ be an $(X,Y)$-straddling palindrome of $W^{(k)}_{n}$. Then $X\rhd (k \oplus W^{(k)}_{n-2k+1})$ and hence $|X|\leq 2^{n-2k+1}$.
\end{lem}
\begin{proof}
{Let $i=n-2k$ and let $W$ be an $(A,B)$-maximal straddling palindrome of $W^{(k)}_{2k+i}$. By (\ref{structeq}),
		\begin{align*}
		W^{(k)}_{2k+i}&=W^{(k)}_{2k+i-1}  \ldots  W^{(k)}_{k+i+1} (k \oplus W^{(k)}_{k+i}),\\
		W^{(k)}_{k+i+1}&=W^{(k)}_{k+i}  \ldots  W^{(k)}_{i+2} (k \oplus W^{(k)}_{i+1}).
		\end{align*}
		Hence,
\begin{align}
		W^{(k)}_{k+i+1} (k \oplus W^{(k)}_{k+i})&\rhd W^{(k)}_{2k+i}, \label{eqrhd2k+i}\\
		W^{(k)}_{i+2} (k \oplus W^{(k)}_{i+1})&\rhd W^{(k)}_{k+i+1}. \label{eqrhdk+i+1}
		\end{align}
Therefore,
\begin{equation}\label{suff3 2k+i}
		W^{(k)}_{i+2} (k \oplus W^{(k)}_{i+1})(k \oplus W^{(k)}_{k+i})\rhd W^{(k)}_{2k+i}.
		\end{equation}
Let $W^{(k)}_{2k+i}= V W^{(k)}_{i+2} (k \oplus W^{(k)}_{i+1})(k \oplus W^{(k)}_{k+i})$ and $l=|V W^{(k)}_{i+2}|$.
By definition of  $(A,B)$-maximal straddling palindrome of $W^{(k)}_{2k+i}$ and using (\ref{suff3 2k+i}), we conclude that
$A\rhd V W^{(k)}_{i+2} (k \oplus W^{(k)}_{i+1})$.
Let $A=  W^{(k)}_{2k+i} [j,j']$, where by definition of $A$, $j'=|W^{(k)}_{2k+i}|-|W^{(k)}_{k+i}|$. We claim that $j>l$; To the contrary, let $j\leq l$.
Let the digits $y$ and $z$ satisfy $yz ={\rm Suff}_2 W^{(k)}_{i+2}$. Since $1\leq i+1\leq k-2$, using Theorem \ref{suffix}, we have $yz=0(i+2)$.
Since $B\lhd k\oplus (W^{(k)}_{k+i})$, we have $0\not\in \mathcal{A}lph(B)$ and we conclude that
 $c_p(W, W^{(k)}_{2k+i})\leq l-2$.
On the other hand by (\ref{suff3 2k+i}), the subsequent letter of $z=i+2$ is $k$. Hence, $(i+2)k \prec W$ and $k (i+2) \prec W$, but this is impossible by Lemma \ref{ab<wn}.
Therefore, $A\rhd (k \oplus W^{(k)}_{i+1})$. Hence, $|X|\leq|A|\leq |W^{(k)}_{i+1}|=2^{i+1}$.
}\end{proof}

The next lemma is useful to give an upper bound for the size of the word $Y$ in any $(X,Y)$-straddling palindrome of $W^{(k)}_{n}$, $2k\leq n\leq 3k-2$.

\begin{lem}\label{palprefix}
Let $i\geq 0$ and $P$ be palindrome prefix of $W=(i+1)W^{(k)}_{k+i}$. Then the largest digit of $P$ is $i+1$. 	
\end{lem}
\begin{proof}
{
 Let $\ell$ be the largest digit of $P$ and for contrary suppose that
$\ell > i+1$.
It is obvious that $\ell < i+k$. Hence, $\ell+1 \leq i+k$ which yields $W^{(k)}_{\ell+1}\lhd W^{(k)}_{i+k}$.
We note that $\ell+1$, which is the last digit of  $W^{(k)}_{\ell+1}$, does not appear in $P$
Hence,
\begin{equation}\label{eq|P|}
|P|\leq |W^{(k)}_{\ell+1}|.
\end{equation}
 On the other hand
$W^{(k)}_{\ell}\lhd W^{(k)}_{\ell+1}$ and since $i+1<\ell$ and using Lemma \ref{lastdig}, we conclude that the first place that $\ell$ occurs is the last digit of $W^{(k)}_{\ell}$. In other words if we let $m=|W^{(k)}_{\ell}|+1$, then $W[m]=\ell$ and for integer $j< m$ we have $W[j]< \ell$.
 By our assumption $\ell \in \mathcal{A}lph(P)$, therefore $c_p(P, W)\geq m=|W^{(k)}_{\ell}|+1$. Thus $|P|\geq 2|W^{(k)}_{\ell}|+1$ and using (\ref{eq|P|}), we have $|P| \leq |W^{(k)}_{\ell+1}|$. Hence, $2|W^{(k)}_{\ell}|+1 \leq |W^{(k)}_{\ell+1}|$ but this contradicts with Corollary \ref{coreqsize}. Hence our assumption is not true and $\ell \leq i+1$, as desired.
}
\end{proof}

%

\begin{lem}\label{strad1|B|}
Let $2k-1 < n < 3k-2$ and let $S=XY$ be an $(X,Y)$-straddling palindrome of $W^{(k)}_{n}$. Then $Y \lhd (k\oplus W^{(k)}_{n-2k+2})(n-k+2)^{-1}$ and hence $|Y|\leq 2^{n-2k+2} -1$.
\end{lem}
\begin{proof}
{Let $i=n-2k$ and $W$ be an $(A,B)$-maximal straddling palindrome of $W^{(k)}_{2k+i}$. Using (\ref{structeq}) we have
 \begin{equation}
	W^{(k)}_{2k+i}=W^{(k)}_{2k+i-1}  \ldots  W^{(k)}_{k+i+1} (k \oplus W^{(k)}_{k+i}).
	\end{equation}
 So we have
 \begin{equation}\label{eq1}
 B\lhd (k\oplus W^{(k)}_{k+i}).
\end{equation}
 By Lemma \ref{strad1|A|},
 \begin{equation}\label{eq2}
 A\rhd (k\oplus W^{(k)}_{i+1}).
 \end{equation}
  and by Lemma \ref{lastdig}, $k+i+2\not\in \mathcal{A}lph(A)$.
 We claim that $k+i+2\not\in \mathcal{A}lph(B)$. For contrary suppose that $k+i+2\in \mathcal{A}lph(B)$.
 We note that $(k\oplus W^{(k)}_{i+2})\lhd (k\oplus W^{(k)}_{k+i})$.
 Now, let $m=| W^{(k)}_{i+2}|$, then by Lemma \ref{lastdig}, $B[m]=k+i+2$ and for all $m'<m$, $B[m']<k+i+2$. On the other hand by Lemma \ref{strad1|A|}, $k+i+2\not\in \mathcal{A}lph(A)$, hence $c_p(W,W)>|A|$.
 Therefore, $(k+i+1)B$
 contains a palindrome $P$ with $k+i+2\in \mathcal{A}lph(P)$, which is impossible by Lemma \ref{palprefix}. Hence,
 $k+i+2\not\in \mathcal{A}lph(B)$, which implies that $B \lhd (k\oplus W^{(k)}_{i+2}(i+2)^{-1})$ and $|B|\leq  2^{i+2} -1=2^{n-2k+2} -1$.}

\end{proof}
Now we are ready to prove the following lemma.
\begin{lem}\label{2k-1<n<3k-2}
	Let $2k-1 < n < 3k-2$. Then $W^{(k)}_{n}$  has exactly two maximal straddling palindromes, one of which is the $(k\oplus W^{(k)}_{n-2k+1}, k\oplus (W^{(k)}_{n-2k+1}(n-2k+1)^{-1}))$-straddling palindrome and the other the $(k\oplus (W^{(k)}_{n-2k+1}),k\oplus (W^{(k)}_{n-2k+1} W^{(k)}_{n-2k+1}(n-2k+1)^{-1}))$-straddling palindrome.
\end{lem}

\begin{proof}
	{
Let $S=AB$ be an $(A,B)$-maximal straddling palindrome of $W^{(k)}_{n}$.
Let $W= k\oplus (W^{(k)}_{i+1}  W^{(k)}_{i+2}(i+2)^{-1})$. Then by Lemmas \ref{strad1|A|} and \ref{strad1|B|}, we conclude that $S \prec W$.
			Moreover, by (\ref{structeq}) we have
		\begin{align*}
		W &= k\oplus (W^{(k)}_{i+1}  W^{(k)}_{i+2}(i+2)^{-1}) \\
		&= k\oplus (W^{(k)}_{i+1} W^{(k)}_{i+1} W^{(k)}_{i} \ldots  W^{(k)}_{0} )\\
		&= k\oplus (W^{(k)}_{i+1} W^{(k)}_{i+1} W^{(k)}_{i+1} (i+1)^{-1})\\
		&= k\oplus (W^{(k)}_{i+1}(i+1)^{-1} (i+1) W^{(k)}_{i+1}(i+1)^{-1} (i+1) W^{(k)}_{i+1} (i+1)^{-1})
		\end{align*}
		%
By the last equation and using Lemma \ref{lastdig}, we conclude that $|W|_{k+i+1}=2$. Let $c_1$ and $c_2$ be two integers with $c_1<c_2\leq |W|$
such that $W[c_1]=W[c_2]=k+i+1$.
		Since $S$ is a straddling palindrome of $W^{(k)}_{2k+i}$, the first occurrence of the digit $k+i+1$ in $W$, lies in $S$. Therefore,
either $c_p(S,W)=c_1$ or $c_p(S,W)=\frac{c_1+c_2}{2}$.
In other words, the center position of $S$ is either the corresponding position of first occurrence of $k+i+1$, or is the position of the middle digit between the two occurrences of $k+i+1$ in $W$.
 So, we have the following two cases
		\begin{itemize}
			\item $S$ has only one digit $k+i+1$: in this case we show that $A=k\oplus W^{(k)}_{i+1}$ and $B=k\oplus (W^{(k)}_{i+1}(i+1)^{-1})$. By Lemma \ref{strad1|A|}, it suffices to show that  $k\oplus( W^{(k)}_{i+1}   W^{(k)}_{i+1}(i+1)^{-1})$ is palindrome and this true by Lemma \ref{wn<k}.
			\item $S$ has exactly two digits $k+i+1$: in this case we show that
			$A=k\oplus W^{(k)}_{i+1}$ and
			$B=k\oplus (W^{(k)}_{i+1} W^{(k)}_{i+1}(i+1)^{-1})$. By Lemma \ref{wn<k},  $k\oplus (W^{(k)}_{i+1}   W^{(k)}_{i+1} W^{(k)}_{i+1}(i+1)^{-1})$ is palindrome and using  Lemmas \ref{strad1|A|} and \ref{strad1|B|}, $S=AB$ is a maximal straddling palindrome, as desired.
		\end{itemize}}
\end{proof}
\begin{lem}\label{strad2|A|}
Let $ n = 3k-2$ and let $S=XY$ be an $(X,Y)$-straddling palindrome of $W^{(k)}_{n}$. Then $X\rhd (k \oplus W^{(k)}_{k-1})$ and hence $|X|\leq 2^{k-1}$.
\end{lem}
\begin{proof}
{Let $W$ be a $(A,B)$-maximal straddling palindrome of $W^{(k)}_{3k-2}$. By Lemma \ref{struct}, we have
		\begin{align*}
		W^{(k)}_{3k-2}&=W^{(k)}_{3k-3}  \ldots  W^{(k)}_{2k-1} (k \oplus W^{(k)}_{2k-2}),\\
		W^{(k)}_{2k-1}&=W^{(k)}_{2k-2}  \ldots  W^{(k)}_{k} (k \oplus W^{(k)}_{k-1}).
		\end{align*}
	 Hence,
	$W^{(k)}_{k} (k \oplus W^{(k)}_{k-1})(k \oplus W^{(k)}_{2k-2})\rhd W^{(k)}_{3k-2}$.
	Since $1 k\rhd W^{(k)}_{k}$ and
	by Lemma \ref{ab<wn}, $k1\not\prec W^{(k)}_{3k-2}$, we conclude that $|A|\leq |k (k\oplus W^{(k)}_{k-1})|=2^{k-1}+1$.
	Moreover, $k k \rhd k (k\oplus W^{(k)}_{k-1})$, but $k k$ could not happen in $k\oplus (W^{(k)}_{k-1} W^{(k)}_{2k-2})$, because otherwise $0 0 \prec (W^{(k)}_{k-1} W^{(k)}_{2k-2})$ which is impossible by Lemma \ref{00}. Therefore,
	 $A\rhd (k\oplus W^{(k)}_{k-1}) $ and $|X|\leq|A|\leq 2^{k-1}$.
}
\end{proof}

\begin{lem}\label{strad2|B|}
Let $n =3k-2$ and let $S=XY$ be an $(X,Y)$-straddling palindrome of $W^{(k)}_{n}$. Then $B \lhd k\oplus (W^{(k)}_{k}k^{-1})$, hence, $|Y|\leq 2^k -1$.
\end{lem}
\begin{proof}
	{Let $W$ be a $(A,B)$-maximal straddling palindrome of $W^{(k)}_{3k-2}$.
		Using (\ref{structeq}) we have	
	\begin{align*}
W^{(k)}_{3k-2}&=W^{(k)}_{3k-3}  \ldots  W^{(k)}_{2k-1} (k \oplus W^{(k)}_{2k-2}),\\
W^{(k)}_{2k-1}&=W^{(k)}_{2k-2}  \ldots  W^{(k)}_{k} (k \oplus W^{(k)}_{k-1}).
\end{align*}
So we have
\begin{equation}\label{eq1}
B\lhd (k\oplus W^{(k)}_{2k-2}).
\end{equation}
By Lemma \ref{strad1|A|},
\begin{equation}\label{eq2}
A\rhd (k\oplus W^{(k)}_{k-1}).
\end{equation}
and by Lemma \ref{lastdig}, $2k\not\in \mathcal{A}lph(A)$.
We claim that $2k\not\in \mathcal{A}lph(B)$ as well.
For contrary suppose that $2k\in \mathcal{A}lph(B)$.
We note that $(k\oplus W^{(k)}_{k})\lhd (k\oplus W^{(k)}_{2k-2})$.
Now, let $m=| W^{(k)}_{k}|$, then by Lemma \ref{lastdig}, $B[m]=2k$ and for all $m'<m$, $B[m']<2k$.
On the other hand by Lemma \ref{strad2|A|}, $2k\not\in \mathcal{A}lph(A)$, hence  $c_p(W,W)>|A|$.
 Therefore, $(2k-1)B$
contains a palindrome $P$ with $2k\in \mathcal{A}lph(P)$, which is impossible by Lemma \ref{palprefix}. Hence,  $2k\not\in \mathcal{A}lph(B)$, which implies that $B \lhd (k\oplus W^{(k)}_{k}(k)^{-1})$ and $|B|\leq  2^{k} -1$.
}
\end{proof}

\begin{lem}\label{n=3k-2}
	Let $n=3k-2$. Then  $W^{(k)}_{n}$ has exactly two maximal straddling palindromes which are respectively $(k\oplus W^{(k)}_{k-1}, k\oplus (W^{(k)}_{k-1}(k-1)^{-1}))$-straddling palindrome and
	$(k\oplus (0^{-1}W^{(k)}_{k-1}), k\oplus (W^{(k)}_{k-1} W^{(k)}_{k-1}( 0 (k-1))^{-1}))$-straddling palindrome in $W^{(k)}_{n}$.
\end{lem}

\begin{proof}
	{In perivous lemma we show that
	$B \lhd (k\oplus (W^{(k)}_{k}k^{-1}))$, so we have
	\begin{align*}
	B &\lhd (k\oplus (W^{(k)}_{k}k^{-1}))\\
	&=k\oplus (W^{(k)}_{k-1}W^{(k)}_{k-2}\ldots W^{(k)}_{1}kk^{-1})\\
	&=k\oplus (W^{(k)}_{k-1}W^{(k)}_{k-2}\ldots W^{(k)}_{1})\\
	&=k\oplus (W^{(k)}_{k-1}W^{(k)}_{k-1}( 0(k-1))^{-1})\\
	\end{align*}	
			Therefore for every $(A,B)$-straddling palindrome of $W^{(k)}_{3k-2}$, $S=A B  \prec W$, where
	\begin{align}
W &= k\oplus (W^{(k)}_{k-1}  W^{(k)}_{k-1}  W^{(k)}_{k-1}( 0(k-1))^{-1} \\
&= k\oplus (W^{(k)}_{k-1}(k-1)^{-1} (k-1)  W^{(k)}_{k-1} (k-1)^{-1} (k-1)  W^{(k)}_{k-1}(0  (k-1))^{-1}). \label{AB}
	\end{align}
By Lemma \ref{lastdig} and (\ref{AB}) the digit $k-1$ occurs twice in $W$. Let $c_1$ and $c_2$ be two integers with $c_1<c_2\leq |W|$ such that $W[c_1]=W[c_2]=k-1$.
Since $S=A B$ is a straddling palindrome of $W^{(k)}_{3k-2}$, $S$ should contain the first occurrence of $k-1$ in $W$. Therefore, either
we have $c_p(S,W)=c_1$ or $c_p(S,W)=\frac{c_1+c_2}{2}$
In other words the center position of $S$ is either the position of the first occurrence of $k-1$ in $W$ or exactly the position of the digit in the middle of the two occurrences of $k-1$ in $W$.
	So, we have the following two cases
	\begin{itemize}
		\item $|S|_{k-1}=1$ : In this case we show that $A=k\oplus W^{(k)}_{k-1}$ and $B=k\oplus (W^{(k)}_{k-1}(k-1)^{-1})$. By Lemma \ref{strad2|A|}, it suffices to show that  $k\oplus( W^{(k)}_{k-1}   W^{(k)}_{k-1}(k-1)^{-1})$ is palindrome and this is true by Lemma \ref{wn<k}.
		\item $|S|_{k-1}=2$ : In this case we show that $A=k\oplus (0^{-1}W^{(k)}_{k-1})$ and
		 $B=k\oplus (W^{(k)}_{k-1} W^{(k)}_{k-1}(0  (k-1))^{-1})$.By Lemma \ref{wn<k}, $k\oplus (0^{-1}W^{(k)}_{k-1}   W^{(k)}_{k-1} W^{(k)}_{k-1}(0  (k-1))^{-1})$ is palindrome and using Lemmas \ref{strad2|A|} and \ref{strad2|B|}, $S=AB$ is a maximal straddling palindrom as desired.
	\end{itemize}}
\end{proof}

\begin{thm}\label{strathm}
Let $S^{(k)}(n)$ be the number of straddling palindromes of  $W_n^{(k)}$. Then
	\begin{equation*}
S^{(k)}(n)= \left\{
\begin{array}{ll}
2^{n-2k+2}-1 & \text{if } 2k-1 \leq n < 3k-2, \\
2^{k}-2 & \text{if }  n = 3k-2, \\
0& \text{otherwise } .
\end{array} \right.
\end{equation*}
\end{thm}

\begin{proof}
{If $n=2k-1$, then by Lemma \ref{n=2k-1}, the only straddling palindrome of $W_n^{(k)}$ is $k k$. Hence, $S^{(k)}(2k-1)=1$.
	
	If $2k \leq n \leq 3k-3$, then by Lemma  \ref{2k-1<n<3k-2}, the only maximal straddling palindromes of  $W_n^{(k)}$ are $S_1=k\oplus W^{(k)}_{i+1}   k\oplus (W^{(k)}_{i+1}(i+1)^{-1})$ and
	$S_2=k\oplus (0^{-1}W^{(k)}_{i+1})   k\oplus (W^{(k)}_{i+1}   W^{(k)}_{i+1}(i+1)^{-1})$, where $i=n-2k$.
	According to the proof of Lemma \ref{2k-1<n<3k-2}, the center of $S_1$ is $k+i+1$ and $c_p(S_1,S_1)=|W^{(k)}_{i+1}|$.
Therefore, the number of straddling palindromes with the same center position as $S_1$, is $|W^{(k)}_{i+1}|-1=2^{i+1}-1=2^{n-2k+1}-1$. Now, again by the proof of Lemma \ref{2k-1<n<3k-2}, $c_p(S_1,S_1)>|W^{(k)}_{i+1}|$.
Hence, the number of straddling
palindromes with the same center position as $S_2$, is $|W^{(k)}_{i+1}|=2^{i+1}=2^{n-2k+1}$.
	Therefore, for $2k \leq n \leq 3k-3 $, $S^{(k)}(n)=2^{n-2k+1}+2^{n-2k+1}-1=2^{n-2k+2}-1$.

	If $n=3k-2$, then by Lemma \ref{n=3k-2}, the only maximal straddling palindromes of  $W_n^{(k)}$ are $S_1=k\oplus W^{(k)}_{k-1}   k\oplus (W^{(k)}_{k-1}(k-1)^{-1})$ and
	$S_2=k\oplus (0^{-1}W^{(k)}_{k-1})   k\oplus (W^{(k)}_{k-1} W^{(k)}_{k-1}(0  (k-1))^{-1})$.
	According to the proof of Lemma \ref{n=3k-2},
$c_p(S-1,S-1)=| W^{(k)}_{k-1}|$.
 Therefore, the number of straddling palindromes with the same center position as $S_1$ is $|W^{(k)}_{k-1}|-1=2^{k-1}-1$. Now, again by the proof of Lemma \ref{n=3k-2},
 $c_p(S_2,S_2)\geq | W^{(k)}_{k-1}|$.
  Hence, the number of straddling palindrome
 with the same center position as $S_2$, is $|0^{-1}W^{(k)}_{k-1}|=2^{k-1}-1$.
 Therefore, $S^{(k)}(3k-2)=2^{k}-2$.
	
	Finally, in the cases $n<2k-1$ of $n>3k-2$ by Lemma \ref{straddling}, we have
	$S^{(k)}(n)=0$, as desired. }
\end{proof}

\subsection{Proof of Theorem \ref{recppthm}}\label{secproofthmpal}

\begin{description}
	\item[(i)] This part is true according to the Theorem \ref{initialthm}.
	\item[(ii)]To prove this part by the equation \ref{recpbs} we have
	\begin{equation}
	P^{(k)}(n)= \Sigma_{i=n-k}^{n-1} P^{(k)}(i)+ \Sigma_{j=n-k+2}^{n-1} B^{(k)}(n,j)+S^{(k)}(n).
	\end{equation}
	Therefore,
		\begin{equation}
	P^{(k)}(n)= \Sigma_{i=n-k}^{n-1} P^{(k)}(i)+\alpha^{(k)}(n),
	\end{equation}
	Where $\alpha^{(k)}(n)= \Sigma_{j=n-k+2}^{n-1} B^{(k)}(n,j)+S^{(k)}(n).$
	By Theorem \ref{strathm}, when $k\leq n\leq 2k-3$, $S^{(k)}(n)=0$. Therefore, $	\alpha^{(k)}(n)= \Sigma_{j=n-k+2}^{n-1} B^{(k)}(n,j)$, when $k\leq n\leq 2k-3$. Now, by Lemma \ref{bordernumber} we have
	\begin{align*}
	\alpha^{(k)}(n)=& \Sigma_{j=n-k+2}^{n-1} B^{(k)}(n,j)\\
	=& \Sigma_{j=n-k+2}^{k-1} B^{(k)}(n,j)\,\,\, {\text{By Lemma \ref{borderpal1}}}\\
	=& \Sigma_{j=n-k+2}^{k-1} (2^j-2^{n-k+1})\\
	=& 2^k+(k-3) 2^{n-k+2}-n2^{n-k+1}.
	\end{align*}
	If $2k-1\leq n\leq 3k-2$, then by Lemma \ref{bordernumber} and Equation \ref{recpbs}, $\alpha^{(k)}(n)=S^{(k)}(n)$. Hence using Theorem \ref{strathm}, we have
	\begin{equation*}
	\alpha^{(k)}(n)= \left\{
	\begin{array}{ll}
	2^{n-2k+2}-1& \text{if} \,\,\,\, 2k-1 \leq n \leq 3k-3,\\
	2^{k}-2& \text{if} \,\,\,\, n=3k-2.
	\end{array} \right.
	\end{equation*}
	Finally, using Theorem \ref{strathm} and Lemma \ref{bordernumber},
	$\alpha^{(k)}(n)=0$, if either $n=2k-2$ or $n>3k-2$.\hfill{$\Box$}
\end{description}

\subsection{Examples}
In the following example in the case $k=4$ for some different values of $n$ we give all of the maximal straddling palindromes and the maximal bordering palindromes of $W^{(k)}_n$, if there exists any.
\begin{example}
Let $k=4$. Then by Lemmas \ref{centerpal1} and \ref{borderpal1},  the word $W^{(4)}_4$ has some bordering palindromes with center $3$ and $2$ respectively, the largest one of these palindromes are shown in (\ref{W443},\ref{W442}). We also notice that by Theorem \ref{strathm}, $W^{(4)}_4$ has no straddling palindrome.
		\begin{align}\label{W443}	
	W^{(4)}_4=&0{\color{red}{102010}}{\color{blue}{3}}{\color{red}{.0102.01}}.4\\
	W^{(4)}_4=&01020103.0{\color{red}{10}}{\color{blue}{2}}{\color{red}{.01}}.4.\label{W442}	
	\end{align}
In the case $n=5= 2k-3$,  by Lemmas \ref{centerpal1} and \ref{borderpal1},  the word $W^{(4)}_5$ has some bordering palindromes with center $3$, the largest one of these palindromes is shown in (\ref{W45}). We also notice that by Theorem \ref{strathm}, $W^{(4)}_5$ has no straddling palindrome.
\begin{equation}\label{W45}
	W^{(4)}_5=010201030102014.01{\color{red}{02010}}{\color{blue}{3}}{\color{red}{.0102}}.45
	\end{equation}
In the case $n=6$, by Theorem \ref{strathm} and Lemma \ref{bordernumber}, $W^{(4)}_6$ has neither straddling palindrome nor bordering palindrome.
	$$W^{(4)}_6=01020103010201401020103010245.010201030102014.01020103.4546$$
\end{example}
\begin{example}	
	Let $k=4$. again by the same reasoning as the previous example the word $W^{(4)}_7$ has one straddling palindrome and no bordering palindrome. Also $W^{(4)}_9$ has two maximal straddling palindromes which are colored below and no bordering palindrome.
\begin{align*}
	W^{(4)}_7= &01020103010201401020103010245010201030102014010201034546.\\
	&01020103010201401020103010245.01020103010201{\color{red}{4.4}}5464547
	\end{align*}
\begin{align*}
	W^{(4)}_9=&01020103010201401020103010245010201030102014010201034546010201030102014\\
	&01020103010245010201030102014454645470102010301020140102010301024501020\\
	&103010201401020103454601020103010201401020103010245454645474546458.0102\\
	&01030102014010201030102450102010301020140102010345460102010301020140102\\
	&010301024501020103010201445464547.0102010301020140102010301024501020103\\
	&010201401020103{\color{blue}{454}}{\color{red}{6}}{\color{blue}{.454}}64547454645845464547454689
		\end{align*}
	\begin{align*}
	W^{(4)}_9=&01020103010201401020103010245010201030102014010201034546010201030102014\\
	&01020103010245010201030102014454645470102010301020140102010301024501020\\
	&103010201401020103454601020103010201401020103010245454645474546458.0102\\
	&01030102014010201030102450102010301020140102010345460102010301020140102\\
	&010301024501020103010201445464547.0102010301020140102010301024501020103\\
	&010201401020103{\color{blue}{4546.4}}{\color{red}{5}}{\color{blue}{46454}}7454645845464547454689.
	\end{align*}
\end{example}

\section{Palindrome Structure}
In the previous section we counted the total number of palindromes in $W^{(k)}_n$ and we gave a recurrence formula for computing this number. In this section firstly we give the structure of all palindrome factors of $W^{(k)}$. Then we compute the palindrome complexity of $W^{(k)}$.

The following lemma indicate that to computing $\mathcal{P}al(W^{(k)})$ it suffices to compute $\mathcal{P}al(W^{(k)}_n)$ for some finite numbers of $n$.

\begin{lem}
For all integer $k>2$, the following relation holds:
$$\mathcal{P}al(W^{(k)})=ki\oplus({\displaystyle\bigcup_{j=1}^{3k-2}}\mathcal{P}al(W^{(k)}_j)).$$
\end{lem}
\begin{proof}
{Using Lemma \ref{bordernumber} and Theorem \ref{strathm}, for $n> 3k-2$, the word $W^{(k)}_n$ has neither straddling palindrome nor bordering palindrome. Hence, by equation (\ref{structeq}), when $n>3k-2$,
$\mathcal{P}al(W^{(k)}_n)=\displaystyle{\bigcup_{j=n-k+1}^{k-1}}\mathcal{P}al(W^{(k)}_j)\bigcup k\oplus(\mathcal{P}al(W^{(k)}_{n-k}))$. Therefore it is easy to see that $\mathcal{P}al(W^{(k)})=ki\oplus({\displaystyle\bigcup_{j=1}^{3k-2}}\mathcal{P}al(W^{(k)}_j))$.}
\end{proof}

To present the next results we need the following definition.

\begin{definition}\label{defP}
Let $k$ be an integer greater than two, then we define the following set of words:
\begin{align*}
P^{(k)}_1:=\{& ki\oplus (W_n^{(k)}n^{-1}): 2\leq n \leq k-1,\,\,\, i\geq 0\},\\
P^{(k)}_2:=\{& ki\oplus ((W^{(k)}_{j-1}W^{(k)}_{j-2}\ldots W^{(k)}_{n-k+1})^R j(W^{(k)}_{j-1}W^{(k)}_{j-2}\ldots W^{(k)}_{n-k+1})):\\
		& k\leq n\leq 2k-3,\,\,\, n-k+2\leq j\leq k-1,\,\,\, i\geq 0\},\\
P^{(k)}_3:= \{& ki\oplus (W^{(k)}_{n-2k+1} W^{(k)}_{n-2k+1}(n-2k+1)^{-1}), ki\oplus (W^{(k)}_{n-2k+1} W^{(k)}_{n-2k+1}  W^{(k)}_{n-2k+1}(n-2k+1)^{-1}): \\
& 2k \leq n \leq 3k-3, \,\,\, i\geq 1\},\\
  P^{(k)}_4:=\{&ki\oplus (W^{(k)}_{k-1}  W^{(k)}_{k-1}(k-1)^{-1}), ki\oplus (0^{-1}W^{(k)}_{k-1} W^{(k)}_{k-1}  W^{(k)}_{k-1}(0 (k-1))^{-1}), ki\oplus (00): i\geq 1\}.
    \end{align*}
\end{definition}

The following lemmas give the structure of the palindromes of $W^{(k)}_n$, when $n\leq 3k-2$.
\begin{lem}\label{paln<k}
Let $n<k$ and $P$ be a maximal palindromic factor of $W^{(k)}_n$. Then $P\in P^{(k)}_1$.
\end{lem}
\begin{proof}
{By Lemma \ref{wn<k}, $W^{(k)}_{n}(n)^{-1}=W^{(k)}_{n-1}(n-1)^{-1} (n-1) W^{(k)}_{n-1}(n-1)^{-1}$ is a maximal palindrome with center $n-1$. Therefore, it is easy to see that the maximal palindromes appeared in $W^{(k)}_{n}$ are either equals to $W^{(k)}_{n}(n)^{-1}$ or are a maximal palindrome of $W^{(k)}_{n-1}(n-1)^{-1}$. Hence, using induction we can see that the set all maximal palindromes of $W^{(k)}_{n}$ is $\{W_n^{(k)}n^{-1}: 2\leq n \leq k-1\}\subset P^{(k)}_1$.}
\end{proof}

\begin{lem}\label{pal k<n<2k-3}
Let $ k\leq n\leq 2k-3$ and $P$ be a maximal palindromic factor of $W^{(k)}_n$. If $P\notin P^{(k)}_1$, then $P\in P^{(k)}_2$.
\end{lem}
\begin{proof}
{By Theorem \ref{strathm}, $W^{(k)}_n$ has no straddling palindrome. Using equation (\ref{structeq}) and lemma \ref{bj}, we found that the only maximal palindromes of $W^{(k)}_n$ which are not appeared in $P^{(k)}_1$
are from the set $P^{(k)}_2=\{(W^{(k)}_{j-1}W^{(k)}_{j-2}\ldots W^{(k)}_{n-k+1})^R\,\, j\,\,(W^{(k)}_{j-1}W^{(k)}_{j-2}\ldots W^{(k)}_{n-k+1})
: n-k+2\leq j\leq k-1\}$.}
\end{proof}

\begin{lem}\label{pal n=2k-2}
Let $ n=2k-2$ and $P$ be a maximal palindromic factor of $W^{(k)}_n$. Then $P\in (P^{(k)}_1 \cup P^{(k)}_2)$.
\end{lem}
\begin{proof}
{By Lemma \ref{bordernumber} and Theorem \ref{strathm}, $W^{(k)}_n$ has neither straddling palindrome nor bordering palindrome. Hence, using equation (\ref{structeq}), we conclude that
$\mathcal{P}al(W^{(k)}_n)=\displaystyle{\bigcup_{j=n-k+1}^{k-1}}\mathcal{P}al(W^{(k)}_j)\bigcup k\oplus(\mathcal{P}al(W^{(k)}_{n-k}))$. By Lemmas \ref{paln<k} and \ref{pal k<n<2k-3}, we know that $\displaystyle{\bigcup_{j=n-k+1}^{k-1}}\mathcal{P}al(W^{(k)}_j)\bigcup k\oplus(\mathcal{P}al(W^{(k)}_{n-k}))\subseteq (P^{(k)}_1 \cup P^{(k)}_2)$. Therefore $\mathcal{P}al(W^{(k)}_n)\subseteq (P^{(k)}_1 \cup P^{(k)}_2)$, as desired.}
\end{proof}

\begin{lem}\label{pal n=2k-1}
Let $ n=2k-2$ and $P$ be a maximal palindromic factor of $W^{(k)}_n$. Then either $P=kk$ or $P\in (P^{(k)}_1 \cup P^{(k)}_2)$.
\end{lem}
\begin{proof}
{Let $P\notin (P^{(k)}_1 \cup P^{(k)}_2)$. Then by Lemma \ref{pal n=2k-2}, $P$ is not a maximal palindromic factor of any $W^{(k)}_j$, $j<n$. Hence, $P$ is either straddling palindrome or a bordering palindrome. By Lemmas \ref{borderpal1} and \ref{n=2k-1}, $P$ is straddling and $P=kk$.}
\end{proof}

\begin{lem}\label{pal 2k<n<3k-3}
Let $2k\leq n \leq 3k-3$ and $P$ be a maximal palindromic factor of $W^{(k)}_n$. Then $P\in (P^{(k)}_1 \cup P^{(k)}_2 \cup P^{(k)}_3 \cup \{kk\})$.
\end{lem}
\begin{proof}
{Let $P\notin (P^{(k)}_1 \cup P^{(k)}_2 \cup  \{kk\})$. Hence by Lemmas \ref{paln<k}, \ref{pal k<n<2k-3}, \ref{pal n=2k-2} and \ref{pal n=2k-1}, $P$ is not a maximal palindromic factor of any $W^{(k)}_j$, $j<n$. Hence, $P$ is either straddling palindrome or a bordering palindrome. By Lemma \ref{bordernumber}, $W^{(k)}_n$ has no bordering palindrome and by Lemma \ref{2k-1<n<3k-2}, the maximal straddling palindromes of $W^{(k)}_n$ are all from the set $\{k\oplus (W^{(k)}_{n-2k+1} W^{(k)}_{n-2k+1}(n-2k+1)^{-1}), k\oplus (W^{(k)}_{n-2k+1} W^{(k)}_{n-2k+1}  W^{(k)}_{n-2k+1}(n-2k+1)^{-1})\}\subset P^{(k)}_3$, as desired.}
\end{proof}

\begin{lem}\label{pal n=3k-2}
Let $2k\leq n \leq 3k-2$ and $P$ be a maximal palindromic factor of $W^{(k)}_n$. Then $P\in (P^{(k)}_1 \cup P^{(k)}_2 \cup P^{(k)}_3 \cup P^{(k)}_4 )$.
\end{lem}
\begin{proof}
{Let $P\notin (P^{(k)}_1 \cup P^{(k)}_2 \cup P^{(k)}_3 )$. Hence by Lemmas \ref{paln<k}, \ref{pal k<n<2k-3}, \ref{pal n=2k-2}, \ref{pal n=2k-1} and \ref{pal 2k<n<3k-3}, $P$ is not a maximal palindromic factor of any $W^{(k)}_j$, $j<n$. Hence, $P$ is either straddling palindrome or a bordering palindrome. By Lemma \ref{bordernumber}, $W^{(k)}_n$ has no bordering palindrome and by Lemma \ref{2k-1<n<3k-2}, the maximal straddling palindromes of $W^{(k)}_n$ are all from the set $\{k\oplus (W^{(k)}_{k-1}  W^{(k)}_{k-1}(k-1)^{-1}), k\oplus (0^{-1}W^{(k)}_{k-1} W^{(k)}_{k-1}  W^{(k)}_{k-1}(0 (k-1))^{-1})\}\subset P^{(k)}_4$, as desired.}
\end{proof}

\begin{thm}\label{palstruct}
	Let $P$ be a  factor of $W^{(k)}$. Then $P$ is a maximal palindromic factor of $W^{(k)}$ if and only if
	$P\in {\displaystyle\bigcup_{i=1}^{4}} P^{(k)}_i $.
\end{thm}
\begin{proof}{First we want to prove the if part. By Lemma \ref{wn<k}, all members of $P^{(k)}_1$ are palindrome and it is clear that all of them are maximal palindrome. If $P\in P^{(k)}_2$, then by Lemma \ref{bj}, $P$ is a maximal palindrome. If $P\in P^{(k)}_3$, then by Lemma \ref{2k-1<n<3k-2}, $P$ is a maximal palindrome. Finally, all members of $P^{(k)}_4$ are maximal palindrome by lemmas \ref{n=2k-1} and \ref{n=3k-2}.

The only if part is trivial using Lemmas (\ref{paln<k}-\ref{pal n=3k-2}).
}\end{proof}

\begin{example}
Consider the word  $W^{(3)}$. Then by Definition\ref{defP}, the sets $P^{(3)}_1, \ldots, P^{(3)}_4$ are as follows:
\begin{align*}
P^{(3)}_1=& \{3i\oplus (010): i\geq 0\},\\
P^{(3)}_2=& \{3i\oplus (10201): i\geq 0\},\\
P^{(3)}_3=&  \{3i\oplus (010), 3i\oplus (01010): i\geq 1\},\\
P^{(3)}_4=& \{3i\oplus (0102010), 3i\oplus (102010201), 3i\oplus (00): i\geq 1\}.\\
\end{align*}
By  Theorem \ref{palstruct}, $\mathcal{P}al(W^{(3)})={\displaystyle\bigcup_{i=1}^{4}} P^{(3)}_i$.
\end{example}

\begin{example}
	Consider the word  $W^{(4)}$. Then by Definition \ref{defP}, the sets $P^{(4)}_1, \ldots, P^{(4)}_4$ are as follows:
	\begin{align*}
	P^{(4)}_1=& \{4i\oplus (010), 4i\oplus (0102010): i\geq 0\},\\
	P^{(4)}_2=& \{4i\oplus (10201),4i\oplus (1020103010201), 4i\oplus (201030102): i\geq 0\},\\
	P^{(4)}_3=&  \{4i\oplus (010), 4i\oplus (01010), 4i\oplus (0102010),4i\oplus (01020102010): i\geq 1\},\\
	P^{(4)}_4=& \{4i\oplus (010201030102010), 4i\oplus (102010301020103010201), 4i\oplus (00): i\geq 1\}.\\
	\end{align*}
By  Theorem \ref{palstruct}, $\mathcal{P}al(W^{(3)})={\displaystyle\bigcup_{i=1}^{4}} P^{(3)}_i$.
\end{example}

\begin{example}
	Consider the word  $W^{(5)}$. Then by Definition \ref{defP}, the sets $P^{(5)}_1, \ldots, P^{(5)}_4$ are as follows:
	\begin{align*}
	P^{(5)}_1= \{&5i\oplus (010), 5i\oplus (0102010),5i\oplus (010201030102010): i\geq 0\},\\
	P^{(5)}_2= \{&5i\oplus (10201),5i\oplus (1020103010201),5i\oplus (10201030102010401020103010201), \\
	& 5i\oplus (201030102),5i\oplus (2010301020104010201030102),
	5i\oplus (30102010401020103):	i\geq 0\},\\
	P^{(5)}_3=  \{&5i\oplus (010), 5i\oplus (01010), 5i\oplus (0102010), 5i\oplus (01020102010), 5i\oplus (010201030102010),\\
	&  5i\oplus (01020103010201030102010): i\geq 1\},\\
	P^{(5)}_4= \{&5i\oplus (00), 5i\oplus (102010301020104010201030102010401020103010201),\\&5i\oplus (0102010301020104010201030102010)
: i\geq 1\}.
	\end{align*}
By  Theorem \ref{palstruct}, $\mathcal{P}al(W^{(3)})={\displaystyle\bigcup_{i=1}^{4}} P^{(3)}_i$.
\end{example}
\subsection{Length of Palindromes in $W^{(k)}$}
In this section we want to compute all possible values for the lengths of palindromes in  $W^{(k)}$.

\begin{definition}\label{defLP}
Let $k\geq 3$, and for $1\leq i \leq 4$, $P_i$ be the sets defined in Definition \ref{defP}.
For $1\leq i \leq 4$, we define $L({P^{(k)}_i}):=\{|P|:    \alpha P {\alpha}^R\in {P^{(k)}_i}, \,\,\,\,{\rm for\,\, some}\,\,\alpha\in{\mathbb{N}}^* \}$.
\end{definition}

\begin{lem}\label{sizepi}
		For each integer $k\geq 3$
		\begin{align*}
		L({P^{(k)}_1}):=\{& 2i-1: \,  2 \leq i\leq 2^{k-2}\},\\
		L({P^{(k)}_2}):=\{& 2i-1: \,  2 \leq i\leq 2^{k-1}-1\},\\
		L({P^{(k)}_3}):= \{& 2i-1: \,  2 \leq i\leq 3.2^{k-3}\},\\
		L({P^{(k)}_4}):=\{&2,2i-1: \,  2 \leq i\leq 3.2^{k-2}\}.
		\end{align*}
\end{lem}
\begin{proof}
	{We just prove $L({P^{(k)}_1})=\{ 2i-1: \,  2 \leq i\leq 2^{k-2}\}$, the proof of the rest parts are similar to this case.
By Definition \ref{defP}, it is clear that the set $L({P^{(k)}_1})$ just contain odd integers. Again by \ref{defP},
it can be seen that if $t\in L({P^{(k)}_1})$ be an odd number greater than $3$, then $t-2\in L({P^{(k)}_1})$.
So if $l_1$ be the maximum integer in $L({P^{(k)}_1})$, then $L({P^{(k)}_1})=\{3,5,\ldots, l_1\}$.
Therefore, to prove this part it suffices to show that $l_1=2^{k-1}-1$.
By Definition \ref{defP}, we found that
\begin{align*}
l_1&=|W_{k-1}^{(k)}|-1\\
&=f^{(k)}_{2k-1} \,\,\,\,({\rm using\,\, equation\,\, (\ref{size})})\\
&=2^{k-1}-1 \,\,({\rm using\,\,\,\,  (\ref{defkbonum})}).
\end{align*}
}
\end{proof}

\begin{thm}
For every integer $k\geq 3$, the palindrome complexity of $W^{(k)}$ is given by
\begin{equation*}
P_{{\tiny W^{(k)}}}(n) = \left\{
\begin{array}{ll}
\infty \;\; &\text{if } \,\, n \in \{2,3,5,\ldots, 3.2^{k-1}-1\},\\
0 \;\; &\text{otherwise.}
\end{array} \right.
\end{equation*}
\end{thm}
\begin{proof}
	{By Theorem \ref{palstruct} and using Lemma \ref{sizepi}, we found that if $n\not\in \{2,3,5,\ldots, 3.2^{k-1}-1\}$, then there is no palindrome factor in  $W^{(k)}$ of length $n$, in other words $P_{{\tiny W^{(k)}}}(n)=0$.
     If $n\in \{2,3,5,\ldots, 3.2^{k-1}-1\}$, then by Theorem \ref{palstruct} and Lemma \ref{sizepi}, $W^{(k)}$ has at least one palindrome factor of length $n$ or equivalently $P_{{\tiny W^{(k)}}}(n)\neq 0$. By Definition \ref{defP} and \ref{defLP}, it is easy to see that in this case $P_{{\tiny W^{(k)}}}(n)=\infty$.}
\end{proof}

\bibliographystyle{acm}
\bibliography{kbonacci20}

\end{document}